\documentclass[12pt]{article}
\usepackage[utf8]{inputenc}
\usepackage{amsthm}
\usepackage{amsfonts}
\usepackage{amsmath}
\usepackage{epsfig}
\usepackage{txfonts}
\usepackage{url}
\newtheorem{theorem}{Theorem}
\newtheorem{conjecture}{Conjecture}
\newtheorem{lemma}[theorem]{Lemma}
\newtheorem{corollary}[theorem]{Corollary}
\newtheorem{observation}[theorem]{Observation}

\title{Density of $5/2$-critical graphs}

\author{Zdeněk Dvořák\thanks{Charles University, Prague, Czech Republic.
E-mail: {\tt rakdver@iuuk.mff.cuni.cz}.  Supported by project GA14-19503S (Graph coloring and structure) of Czech Science Foundation.}\and
Luke Postle\thanks{University of Waterloo. E-mail: {\tt lpostle@uwaterloo.ca}.  Partially supported by NSERC under Discovery Grant No. 2014-06162.}}
\date{}

\begin{document}
\maketitle

\begin{abstract}
A graph $G$ is \emph{$5/2$-critical} if $G$ has no circular $5/2$-coloring (or equivalently,
homomorphism to $C_5$), but every proper subgraph of $G$ has one.  We prove that every $5/2$-critical
graph on $n\ge 4$ vertices has at least $\frac{5n-2}{4}$ edges, and list all $5/2$-critical
graphs achieving this bound.  This implies that every planar or projective-planar graph
of girth at least $10$ is $5/2$-colorable.
\end{abstract}

\section{Introduction}

Let $G$ and $H$ be graphs.  A \emph{homomorphism} from $G$ to $H$ is a function $f:V(G)\to V(H)$ such that for every $uv\in E(G)$,
the graph $H$ contains the edge $f(u)f(v)$.  In this paper, we are concerned with homomorphisms to a $5$-cycle, which
are intensively studied in the context of circular colorings.

Let $G$ be a graph. For a real number $r\ge 1$, a function $\varphi$ from $V(G)$ to the circle with circumference $r$ is a \emph{circular $r$-coloring}
if for every edge $uv\in E(G)$, the distance between $\varphi(u)$ and $\varphi(v)$ in the circle is at least one.  The infimum of $r\ge 1$ such that $G$ has
an $r$-coloring is called the \emph{circular chromatic number} of $G$ and denoted by $\chi_c(G)$.  Circular chromatic number was introduced by Vince~\cite{vince},
who also showed that the circular chromatic number of a graph differs from the ordinary chromatic number by at most $1$, and thus it refines the information on coloring
properties of the graph.  For more results on circular chromatic number, see the surveys of Zhu~\cite{zhusurvey1,zhusurvey2}.

It is easy to see that for an integer $t\ge 1$, a graph $G$ has circular chromatic number at most $2+\frac{1}{t}$ if and only if $G$ has a homomorphism to $C_{2t+1}$.
Jaeger~\cite{jaeger} gave a conjecture concerning modular orientations, whose dual for planar graphs is the following.
\begin{conjecture}\label{conj-main}
For every integer $t\ge 1$, every planar graph of girth at least $4t$ has a homomorphism to $C_{2t+1}$.
\end{conjecture}

The case $t=1$ is equivalent to Gr\"otzsch' theorem~\cite{grotzsch1959}, and it is the only case where Conjecture~\ref{conj-main} is confirmed.
For general $t$, the best known result is by Borodin et al.~\cite{bkkw} who proved that every planar graph of girth at least $\frac{20t-2}{3}$ has a homomorphism to $C_{2t+1}$.
This improves a previous result of Zhu~\cite{zhucirc}.
In this paper, we focus on the case $t=2$.  Here, Zhu~\cite{zhucirc} shows that every planar graph of girth at least $12$ has a homomorphism to $C_5$.
According to~\cite{devosopen}, DeVos and Deckelbaum claim an unpublished improvement showing that every planar graph without odd cycles of length at most $9$
(and consequently, every planar graph of girth at least $10$) has a homomorphism to $C_5$.

In this paper, we give a strengthening of these results by showing that it is sufficient to only bound the density of the colored graph, rather than assuming its planarity.
A first result in this direction is by Borodin et al.~\cite{dens65}, who proved that if $G$ is a triangle-free graph such that $e(H)/n(H)< 6/5$ for every $H\subseteq G$, then $G$ is $5/2$-colorable.
It would be possible to state our result similarly in the terms of the maximum edge density of subgraphs, but it is more convenient (and slightly less restrictive)
to work in the setting of critical graphs.  A graph $G$ is \emph{$5/2$-critical} if $G$ has no circular $5/2$-coloring (or equivalently,
homomorphism to $C_5$), but
every proper subgraph of $G$ has one.  Let $e(G)$ denote the number of edges of $G$
and $n(G)$ the number of its vertices.  Let the \emph{potential} of $G$ be $p(G)=5n(G)-4e(G)$.
Let $C_k$ denote the cycle with $k$ vertices and $P_k$ the path with $k$ edges.
Note that $p(C_k)=k$ and $p(P_k)=5+k$.

\begin{theorem}\label{thm-mainsimp}
If $G$ is a $5/2$-critical graph distinct from $C_3$, then $p(G)\le 2$.
\end{theorem}

To see how this relates to the previous results, let us state the well-known consequence of Euler's formula for the density of an embedded graph.

\begin{observation}\label{obs-dense}
Let $G$ be graph with a $2$-cell embedding in a surface of Euler genus $g$ such that the boundary walk of every face of $G$ has length at least $10$.
Then $p(G)\ge 10-5g$.
\end{observation}

Consequently, we have the following.
\begin{corollary}\label{cor-girth10}
Every planar or projective-planar graph of girth at least $10$ has a homomorphism to $C_5$.
\end{corollary}
\begin{proof}
Suppose for a contradiction that $G$ is a planar or projective-planar graph of girth at least $10$ that has no homomorphism to $C_5$.
Hence, $G$ contains a $5/2$-critical subgraph $G_1$.  Observe that $G_1$ is connected, and if it is not planar,
then its embedding in the projective plane is cellular.  Furthermore, $G_1$ is not a tree, and thus
the boundary walk of every face of $G_1$ includes a cycle, and by the girth assumption it has length at least $10$.
Consequently, $p(G_1)\ge 10-5g\ge 5$, which contradicts Theorem~\ref{thm-mainsimp}.
\end{proof}

Note that the edge density condition of Theorem~\ref{thm-mainsimp} cannot be strengthened so that it would imply Conjecture~\ref{conj-main} with $t=2$.
A natural bound follows from $6$-critical graphs (i.e., graphs that are not $5$-colorable, but all their proper subgraphs
are $5$-colorable).  If $G$ is a $6$-critical graph, it is easy to see that the graph $G'$ obtained from $G$ by subdividing each edge twice
is $5/2$-critical.  For every integer $m\ge 1$, a construction of Ore~\cite{ore} gives a $6$-critical graph $G$ with $5m+1$ vertices and $14m+1$ edges,
and thus $G'$ has $33m+3$ vertices and $42m+3$ edges, giving asymptotically the edge density $42/33=14/11$.
On the other hand, planar graphs of girth $8$ can have edge density arbitrarily close to $4/3$.  Using the Folding lemma~\cite{KloZhang}, we could eliminate facial $8$-cycles,
bringing the density down to $9/7$, but this is still too large.  Nevertheless, we find it plausible that asymptotically, the density $14/11$ is the right
bound for $5/2$-critical graphs. More strongly, we conjecture the following bound, which is tight for the graphs obtained by the Ore construction as described, as well as the triangle.

\begin{conjecture}\label{conj-genmain}
If $G$ is a $5/2$-critical graph, then $14n(G)-11e(G)\le 9$.
\end{conjecture}

Let us now give a few more consequences of our result.
In the planar case, we can use the well-known Folding lemma to argue that we only need to exclude odd cycles.

\begin{corollary}\label{cor-oddgirth}
Every planar graph without odd cycles of length at most $9$ has a homomorphism to $C_5$.
\end{corollary}
\begin{proof}
Suppose for a contradiction that there exists a planar graph $G$ without odd cycles of length at most $9$ that has no homomorphism to $C_5$.
Choose $G$ with $n(G)+e(G)$ minimal.  In particular, $G$ is $5/2$-critical, and thus by Theorem~\ref{thm-mainsimp} and
Observation~\ref{obs-dense}, a plane embedding of $G$ has a face with boundary walk of length at most $9$.
By the Folding lemma~\cite{KloZhang}, there exist distinct non-adjacent vertices $u,v\in V(G)$ such that if $G_1$ is the graph obtained from $G$ by identifying $u$ with $v$ to a new vertex $z$,
then $G_1$ is planar and has no odd cycle of length at most $9$.  However, by the minimality of $G$, there exists a homomorphism $\psi$ from $G_1$ to $C_5$, which can be extended
to a homomorphism from $G$ to $C_5$ by setting $\psi(u)=\psi(v)=\psi(z)$.  This is a contradiction.
\end{proof}

Finally, the difference between the bounds of Theorem~\ref{thm-mainsimp} and Observation~\ref{obs-dense} is rather large in the planar case, enabling us to obtain a
strengthening with precolored vertices.  Let $G$ be a graph and $C$ a $5$-cycle, and consider $v_1,v_2\in V(G)$ and $c_1,c_2\in V(G)$.  We say that
the pair $(c_1,c_2)$ is \emph{plausible for $(v_1,v_2)$} if
\begin{itemize}
\item $v_1=v_2$ and $c_1=c_2$, or
\item $v_1v_2\in E(G)$ and $c_1c_2\in E(C)$, or
\item the distance between $v_1$ and $v_2$ in $G$ is exactly two, and $c_1$ is not adjacent to $c_2$ in $C$, or
\item the distance between $v_1$ and $v_2$ in $G$ is exactly three, and $c_1\neq c_2$, or
\item the distance between $v_1$ and $v_2$ in $G$ is at least four.
\end{itemize}
Observe that if $G$ has a homomorphism $\psi$ to $C$ such that $\psi(v_1)=c_1$ and $\psi(v_2)=c_2$, then $(c_1,c_2)$ is plausible for $(v_1,v_2)$.
Conversely, we have the following.

\begin{corollary}\label{cor-precol}
Let $G$ be a planar graph of girth at least $10$, let $C$ be a $5$-cycle, and let $v_1,v_2\in V(G)$ and $c_1,c_2\in V(C)$ be arbitrary vertices.
Then $G$ has a homomorphism $\psi$ to $C$ such that $\psi(v_1)=c_1$ and $\psi(v_2)=c_2$ if and only if $(c_1,c_2)$ is plausible for $(v_1,v_2)$.
\end{corollary}
\begin{proof}
Suppose first that $c_1\neq c_2$, and thus $v_1\neq v_2$.
Let $G_1$ be the graph obtained from $G$ and $C$ by identifying $v_1$ with $c_1$, and $v_2$ with $c_2$.  Since $(c_1,c_2)$ is plausible for $(v_1,v_2)$, the graph $G_1$ is triangle-free.
If $G_1$ has a homomorphism $\psi$ to $C$, then without loss of generality we can assume that the restriction of $\psi$ to $C$ is the identity, and the restriction of $\psi$ to $G$
satisfies $\psi(v_1)=c_1$ and $\psi(v_2)=c_2$ as required.

Suppose now for a contradiction that no such homomorphism exists.  Let $G_2$ be a $5/2$-critical subgraph of $G_1$, and let $G_2'=G_2\cap G$ and $G_2''=G_2\cap C$.
By Corollary~\ref{cor-girth10}, $G_2'$ has a homomorphism to $C_5$, and since $G_2$ has no such homomorphism, it follows that $G_2''$ is either equal to $C$ or to a path between $c_1$ and $c_2$ in $C$.
In either case, $p(G''_2)\ge 5$.  If $G'_2$ is not a tree, then each face of an embedding of $G'_2$ in the plane has length at least $10$, and Observation~\ref{obs-dense} implies that 
$p(G'_2)\ge 10$.  However, then $p(G_2)=p(G'_2)+p(G''_2)-p(G'_2\cap G''_2)=p(G'_2)+p(G''_2)-10\ge 5$, which contradicts Theorem~\ref{thm-mainsimp}.

We conclude that $G'_2$ is a tree.  Since $G_2$ is $5/2$-critical, it does not contain vertices of degree one, and thus $G'_2$ is a path joining $v_1$ with $v_2$.  Consequently, $G_2$ is either a cycle or a theta graph.
However, it is easy to check that no cycle other than triangle and no theta graph is $5/2$-critical, which is a contradiction.

The case that $c_1=c_2$ is dealt with similarly, starting with letting $G_1$ be the graph obtained from $G$ by identifying $v_1$ with $v_2$.
\end{proof}

Theorem~\ref{thm-mainsimp} is tight.  For the purposes of the proof, we also need to list the $5/2$-critical graphs whose potential is exactly $2$.
Let $E_1$ be the graph obtained from a $9$-cycle $v_1\ldots v_9$ by adding a vertex adjacent to $v_1$, $v_4$, and $v_7$.
Let $E_2$ be the graph obtained from two intersecting $5$-cycles $uvx_1x_2x_3$ and $uvy_1y_2y_3$ by adding a path
$x_2z_1z_2y_2$.  See Figure~\ref{fig-except}.  Note that both $E_1$ and $E_2$ have $10$ vertices and $12$ edges,
and thus $p(E_1)=p(E_2)=2$.  A stronger form of our result is the following.

\begin{figure}
\centering{\includegraphics{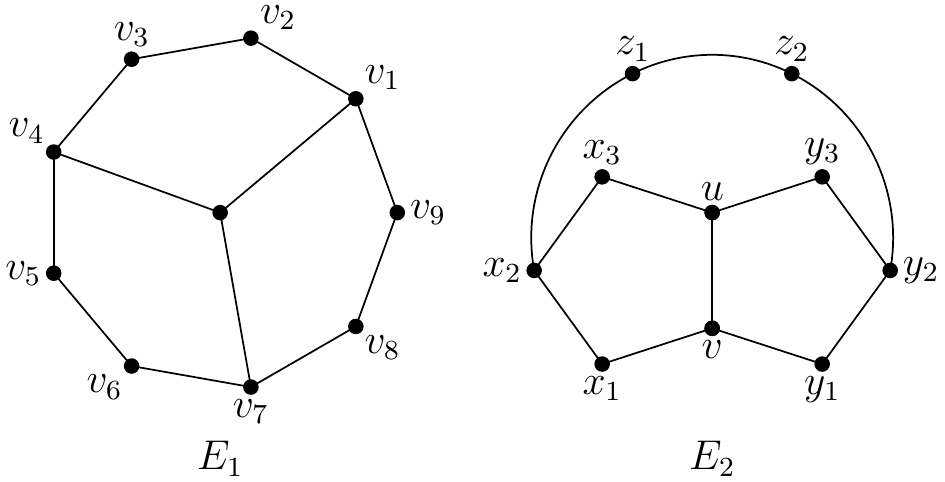}}
\caption{Exceptional graphs.}\label{fig-except}
\end{figure}

\begin{theorem}\label{thm-main}
If $G$ is a $5/2$-critical graph distinct from $C_3$, $E_1$ and $E_2$, then $p(G)\le 1$.
\end{theorem}

In the following sections, we give a proof of Theorem~\ref{thm-main}.  
We say that $G$ is a \emph{smallest counterexample} if $G$ is $5/2$-critical, $G\not\in\{C_3, E_1, E_2\}$, $p(G)\ge 2$
and $p(H)\le 1$ for every $5/2$-critical graph $H\not\in\{C_3, E_1, E_2\}$ satisfying $n(H)<n(G)$.  Clearly, Theorem~\ref{thm-main} holds unless
there exists a smallest counterexample.  

Our proof is heavily influenced by the method of Kostochka and Yancey~\cite{koyan} for bounding the density of
$4$-critical graphs---in Lemma~\ref{lemma-light}, we show that almost any proper subgraph $H$ of a smallest counterexample $G$
has potential at least $6$ (i.e., it is relatively sparse).  Consequently, we can argue that specific graphs derived from $H$ (say by adding an extra edge)
have potential at least $2$, and from the minimality of $G$, we derive that they have a homomorphism to $C_5$.  This is useful in many of the arguments
that prove non-existence of particular configurations in a smallest counterexample.

In Section~\ref{sec-prop}, we establish several structural properties of a hypothetical smallest counterexample,
and in Section~\ref{sec-disch}, we disprove its existence by a discharging argument.  An important idea is to consider charges of vertices in each $5$-cycle
in a smallest counterexample together.  This is made possible by observing that the $5$-cycles in a smallest counterexample are pairwise vertex-disjoint
(Lemma~\ref{lemma-disj5}).  Furthermore, we can argue that the neighborhoods of $5$-cycles are relatively dense (Lemmas~\ref{lemma-wtcell3} and \ref{lemma-wtcell4}).
On the other hand, vertices that do not belong to any $5$-cycle are easier to deal with by themselves, since e.g. identifying some of their neighbors to
a single vertex does not create a triangle.  By exploiting this, we can argue about the density of their neighborhoods (Lemmas~\ref{lemma-opp5c} and its Corollaries~\ref{cor-no22x} and \ref{obs-v3}, Lemmas~\ref{lemma-wt4}, \ref{lemma-111} and \ref{lemma-2111}).

\section{The properties of a smallest counterexample}\label{sec-prop}

For a graph $H$, let $P_n(H)$ denote the set of graphs obtained from $H$ by adding a path of length $n$ joining
two distinct vertices of $H$.  Let $Q(H)$ denote the set of graphs obtained from $H$ by adding a vertex and joining it to three distinct
vertices of $H$ by paths with $k_1$, $k_2$, and $k_3$ edges, where $(k_1,k_2,k_3)\in\{(1,3,3),(2,2,3)\}$.  For $i\in\{1,2\}$, let $E_i(H)$ denote the set of graphs obtained from $H$ and $E_i$ as follows:
select a vertex $v\in V(E_i)$ of degree $k\in\{2,3\}$ and split $v$ to $k$ vertices $v_1$, \ldots, $v_k$ of degree one.
Let $u_1$, \ldots, $u_k$ be vertices of $H$, not all equal.  For $i\in\{1,\ldots, k\}$, identify $u_i$ with $v_i$.
See Figure~\ref{fig-excext}.  Let $X(H)=\{H\}\cup P_2(H)\cup P_3(H)\cup Q(H)\cup E_1(H)\cup E_2(H)$.

\begin{figure}
\centering{\includegraphics{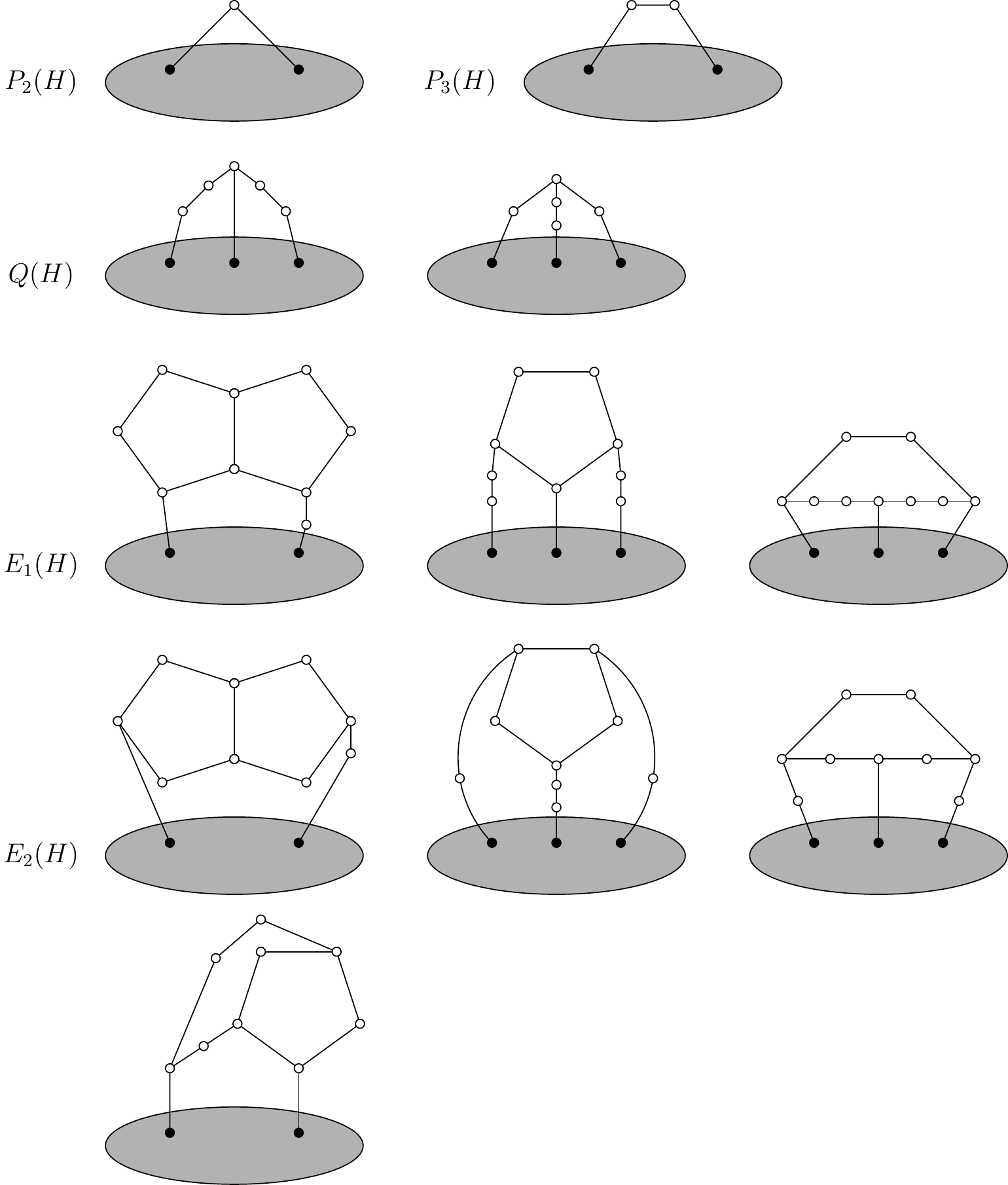}}
\caption{Configurations from Lemma~\ref{lemma-light}.}\label{fig-excext}
\end{figure}

\begin{lemma}\label{lemma-light}
Let $G$ be a smallest counterexample and let $H$ be a subgraph of $G$.  Then
\begin{itemize}
\item $p(H)\ge 2$ if $H=G$,
\item $p(H)\ge 4$ if $G\in P_3(H)$,
\item $p(H)\ge 5$ if $G\in P_2(H)\cup Q(H)\cup E_1(H)\cup E_2(H)$,
\item $p(H)=5$ if $n(H)=1$ or $H=C_5$, and
\item $p(H)\ge 6$ otherwise.
\end{itemize}
\end{lemma}
\begin{proof}
Note that the claim holds when $H\subseteq C_5$ or $H=G$.  Furthermore, $2\le p(G)=p(H)+5(n(G)-n(H))-4(e(G)-e(H))$.  We have
$$(n(G)-n(H),e(G)-e(H))=\begin{cases}
(1,2)&\text{if $G\in P_2(H)$}\\
(2,3)&\text{if $G\in P_3(H)$}\\
(5,7)&\text{if $G\in Q(H)$}\\
(9,12)&\text{if $G\in E_1(H)\cup E_2(H)$,}\\
\end{cases}$$
and thus the claim holds if $G\in X(H)$.

For a contradiction, assume that $G$ contains a proper subgraph $H$ such that $G\not\in X(H)$, $H\not\subseteq C_5$
and $p(H)\le 5$.  Choose such a subgraph with $n(H)+e(H)$ as large as possible.
If $H$ is not an induced subgraph, consider an edge $e\in E(G)\setminus E(H)$ joining two vertices of $H$.
Note that $p(H+e)=p(H)-4<p(H)$, and by the maximality of $H$, we conclude that $G\in X(H+e)$.
Hence, $p(H)=p(H+e)+4\ge 6$.  This is a contradiction, and thus $H$ is an induced subgraph of $G$.

Since $G$ is $5/2$-critical, $n(H)<n(G)$ and $H$ is not a subgraph of a $5$-cycle,
it follows that $H$ has a homomorphism $\psi$ to a subgraph $F$ of a $5$-cycle satisfying $n(F)<n(H)$.
Let $G_1$ be the graph obtained from $G$ by identifying vertices of $H$ with vertices of $F$ according to $\psi$.
If parallel edges arise, for each pair of vertices we delete all but one (arbitrary) edge joining them;
thus $G_1$ is simple and each edge of $G_1$ corresponds to a unique edge of $G$.
Since $G$ has no homomorphism to $C_5$, it follows that $G_1$ has no homomorphism to $C_5$,
and thus it contains a $5/2$-critical subgraph $G_2$.

Suppose first that $G_2$ is a triangle.  It follows that $G$ contains a path $P$ of length $n\in \{2,3\}$ with
endvertices in $H$ and with internal vertices in $V(G)\setminus V(H)$.  
Note that $p(H+P)=p(H)+5(n-1)-4n\le p(H)-2<p(H)$, and thus by the maximality of $n(H)+e(H)$,
we conclude that $G\in X(H+P)$.  
Since $G\not\in P_n(H)$, we have $H+P\neq G$, and thus $p(H+P)\ge 4$.
However, then $p(H)\ge p(H+P)+2\ge 6$, which is a contradiction.

Therefore, $G_2\neq C_3$, and since $G$ is a smallest counterexample, we have $p(G_2)\le 2$, and $p(G_2)\le 1$ if $G_2\not\in\{E_1,E_2\}$.
Let $G_3$ be the graph obtained from $G_2$ by replacing the subgraph $F_2=F\cap G_2$ by $H$ (redirecting the edges which join vertices of $V(G_2)\setminus V(F_2)$
with $F_2$ to the appropriate vertices of $H$, but not adding any other edges between $V(H)$ and $V(G_2)\setminus V(F_2)$).
Note that $F_2$ is non-empty, since $G$ cannot contain a proper $5/2$-critical subgraph.
Since $F_2$ is a subgraph of a $5$-cycle, we have $p(F_2)\ge 5$, and $p(F_2)\ge 6$ unless $n(F_2)=1$ or $F_2=C_5$.
Observe that $p(G_3)=p(G_2)-p(F_2)+p(H)\le p(H)-3<p(H)$.  By the minimality of $H$, we conclude that $G\in X(G_3)$.
However, then $p(G_3)\ge 2$, and $p(G_3)\ge 4$ unless $G=G_3$.

Since $5\ge p(H)=p(G_3)+p(F_2)-p(G_2)$, it follows that $G=G_3$,
$n(F_2)=1$ or $F_2=C_5$, and $G_2\in\{E_1,E_2\}$.  However, if $F_2=C_5$, then this implies that $G\in Q(H)$, and
if $n(F_2)=1$, then $G\in E_1(H)\cup E_2(H)$.  This is a contradiction.
\end{proof}

\begin{lemma}\label{lemma-girth}
Each smallest counterexample has girth at least $5$.
\end{lemma}
\begin{proof}
Suppose for a contradiction that a smallest counterexample $G$ contains a cycle $K$ of length at most $4$.  Since $G\neq C_3$ and $G$ is $5/2$-critical,
it follows that $K$ is a $4$-cycle and $G\neq K$.  However, $p(K)=4$, and thus Lemma~\ref{lemma-light} implies that
$G\in P_3(K)$.  However, no such graph is $5/2$-critical.
\end{proof}

Throughout the paper, we often need to show that graphs derived from $E_1$ and $E_2$ by various operations (e.g., splitting vertices
and attaching specific subgraphs) are not smallest counterexamples.  While it is possible in each case to perform the necessary case
analysis by hand (and indeed, we did that when producing this proof), writing out all the details would make the paper substantially
longer without giving the reader any additional insight.  Hence, we opt to exclude such exceptional graphs at once using computer search, instead.

\begin{lemma}\label{lemma-nosmall}
If $G$ is a smallest counterexample and $p(G)=2$, then $G$ has at least $22$ vertices.
\end{lemma}
\begin{proof}
Suppose that $G$ is a smallest counterexample with at most $21$ vertices and $p(G)=2$,
i.e., $e(G)=\frac{5n(G)-2}{4}$.  By integrality, it follows that $n(G)\in\{6,10,14,18\}$.
By Lemma~\ref{lemma-girth}, $G$ has girth at least $5$.  Furthermore, since $G$ is $5/2$-critical,
it is $2$-connected.  Using the program \textbf{geng} of Brendan McKay, we enumerated all such graphs,
and verified that the only $5/2$-critical graphs satisfying the conditions are $E_1$ and $E_2$.
The source code of the program we used can be found at \url{http://atrey.karlin.mff.cuni.cz/~rakdver/circul54.c}.
\end{proof}

We also need the following well-known fact.

\begin{observation}\label{obs-mincyc}
If a graph $H$ contains at least two distinct cycles, but every proper subgraph of $H$ contains
at most one cycle, then $H$ is either a union of two cycles intersecting in at most one vertex,
or a theta graph.
\end{observation}

Using this fact, we can now restrict intersections of cycles in a smallest counterexample.

\begin{lemma}\label{lemma-disj5fv}
If $H$ is a subgraph of a smallest counterexample containing at least two distinct cycles, then $n(H)\ge 9$.  In particular,
every two $5$-cycles in a smallest counterexample are edge-disjoint.
\end{lemma}
\begin{proof}
Let $G$ be a smallest counterexample, and suppose that $H\subseteq G$ contains at least two distinct cycles.
Choose $H$ so that $n(H)$ is minimum, and subject to that, $e(H)$ is minimum.
If $H$ contains two cycles intersecting in at most one vertex, then $n(H)\ge 9$ by Lemma~\ref{lemma-girth} as desired.
Hence, by Observation~\ref{obs-mincyc}, we can assume that $H$ is a theta graph.  No theta graph is $5/2$-critical, and thus $H\neq G$.
Note that $e(H)=n(H)+1$, and thus $p(H)=n(H)-4$.  By Lemma~\ref{lemma-light}, we have $n(H)\ge 8$.

Suppose for a contradiction that $n(H)=8$.
Then $p(H)=4$, and by Lemma~\ref{lemma-light}, we have $G\in P_3(H)$.  Since $n(G)=n(H)+2$ and $e(G)=e(H)+3$,
we have $p(G)=p(H)+2\cdot 5-3\cdot 4=2$.  Also, $n(G)=10$, and thus $G$ contradicts Lemma~\ref{lemma-nosmall}.
\end{proof}

A \emph{string} in a graph $G$ is a path with internal vertices of degree two and endvertices of degree at least $3$.  Note that an edge joining two
vertices of degree at least three is also a string.
For an integer $k\ge 0$, a \emph{$k$-string} is a string with $k+1$ edges.
If $u$ and $v$ are the endvertices of a string, then we say that $v$ is a \emph{friend} of $u$.

Let $P$ be a $k$-string with $k\ge 3$, let $u$ and $v$ be the endvertices of $P$ and let $C$ be a $5$-cycle.
Observe that every function $\psi:\{u,v\}\to V(C)$ can be extended to a homomorphism from $P$ to $C$.  Consequently,
we obtain the following claim.

\begin{observation}\label{obs-simp}
If $G$ is a $5/2$-critical graph, then $G$ has minimum degree at least two and contains no $k$-strings with $k\ge 3$.
\end{observation}

Cycles of length $5$ will play an important role in our proof.  Let a $5$-cycle in a smallest counterexample be called a \emph{cell}.
Observation~\ref{obs-simp} and Lemma~\ref{lemma-disj5fv} imply that in any smallest counterexample, if $K$ is a cell
and $P\not\subset K$ is a string, then $P$ has at most one end in $K$.
The \emph{degree} of $K$, denoted by $\deg(K)$, is the number of strings $P\not\subset K$ that intersect $K$.
Let us now give key lemmas establishing the importance of cells.

\begin{lemma}\label{lemma-opp5c}
Let $G$ be a smallest counterexample, let $v\in V(G)$ have degree $3$, and suppose that $v$ is an end of a $2$-string
$vv_1xy$.  Let $v_2$ and $v_3$ be the neighbors of $v$ distinct from $v_1$.  Then the path $v_2vv_3$ is contained
in a cell.
\end{lemma}
\begin{proof}
Suppose for a contradiction that the path $v_2vv_3$ is not contained in a cell.  Let $G_1$ be the graph obtained from $G-\{v,v_1,x\}$ by identifying $v_2$ with $v_3$
to a new vertex $z$.  Note that $G_1$ is triangle-free.

Suppose that there exists a homomorphism $\psi$ of $G_1$ to a $5$-cycle $C=c_1\ldots c_5$, where without loss of generality $\psi(z)=c_1$ and
$\psi(y)\in \{c_1,c_2,c_3\}$.  We can extend $\psi$ to a homomorphism from $G$ to $C$ as follows.  We set $\psi(v_2)=\psi(v_3)=c_1$.
If $\psi(y)\in \{c_1,c_3\}$, then set $\psi(v)=\psi(x)=c_2$ and $\psi(v_1)=c_1$.  If $\psi(y)=c_2$, then set
$\psi(v)=c_5$, $\psi(v_1)=c_4$ and $\psi(x)=c_3$.
Since $G$ is $5/2$-critical, this is a contradiction, and thus $G_1$ has no homomorphism to a $5$-cycle.

Consequently, $G_1$ contains a $5/2$-critical subgraph $G_2$.  Since $G$ is $5/2$-critical, note that $G_2\not\subset G$,
and thus $z\in V(G_2)$.  Let $G_3$ be the graph obtained from $G_2$ by splitting $z$ back to $v_2$ and $v_3$
and adding the path $v_2vv_3$.
Since $G$ is a smallest counterexample, we have $p(G_2)\le 2$, and $p(G_2)\le 1$ unless $G_2\in\{E_2,E_3\}$.  Furthermore, $p(G_3)=p(G_2)+5(n(G_3)-n(G_2))-4(e(G_3)-e(G_2))=p(G_2)+2\le 4$.
Since $v_1\not\in V(G_3)$, we have $G\neq G_3$, and thus Lemma~\ref{lemma-light} implies that $G\in P_3(G_3)$ and $p(G_2)=2$, i.e., $G_2\in \{E_1,E_2\}$.  Hence, $p(G_3)=4$ and $p(G)=p(G_3)+2\cdot 5-3\cdot 4=2$.
Since $n(G)=n(G_3)+2=n(G_2)+4=14$, this contradicts Lemma~\ref{lemma-nosmall}.
\end{proof}

Now we investigate coloring properties of cells of degree three.

\begin{lemma}\label{lemma-ext3cell}
Let $K=v_1\ldots v_5$ be a $5$-cycle in a triangle-free graph $G$ such that $v_1$, $v_2$, and $v_4$ have degree three
and $v_3$ and $v_5$ have degree two.  Let $u_1$, $u_2$, and $u_4$ be the neighbors of $v_1$, $v_2$, and $v_4$, respectively,
not contained in $K$.  Let $\psi$ be a homomorphism from $G-V(K)$ to a $5$-cycle $C$.
If $\psi(u_1)$ is adjacent to $\psi(u_2)$ in $C$, then $\psi$ extends to a homomorphism from $G$ to $C$.
\end{lemma}
\begin{proof}
Let $C=c_1\ldots c_5$.  Without loss of generality,
$\psi(u_1)=c_1$ and $\psi(u_2)=c_2$.  Choose $\psi(v_4)\in \{c_3,c_4,c_5\}$ adjacent to $\psi(u_4)$.
If $\psi(v_4)=c_3$, then let $\psi(v_1)=c_5$, $\psi(v_2)=c_1$, $\psi(v_3)=c_2$ and $\psi(v_5)=c_4$.
If $\psi(v_4)=c_4$, then let $\psi(v_1)=c_2$, $\psi(v_2)=c_1$, $\psi(v_3)=c_5$ and $\psi(v_5)=c_3$.
If $\psi(v_4)=c_5$, then let $\psi(v_1)=c_2$, $\psi(v_2)=c_3$, $\psi(v_3)=c_4$ and $\psi(v_5)=c_1$.
In all cases, we obtain a homomorphism from $G$ to $C$.
\end{proof}

\begin{lemma}\label{lemma-wtcell3}
A smallest counterexample does not contain a cell of degree at most three.
\end{lemma}
\begin{proof}
For a contradiction, suppose that $K=v_1v_2v_3v_4v_5$ is a cell of degree at most three in a smallest counterexample $G$. 
By Observation~\ref{obs-simp}, $G$ does not contain any $k$-strings with $k\ge 3$.
If $G$ contained a $2$-string $P\subseteq K$, then at least one end of $P$ would have degree three,
and Lemma~\ref{lemma-opp5c} would imply that $G$ contains two $5$-cycles sharing an edge, contrary to Lemma~\ref{lemma-disj5fv}.

Therefore, we can assume that $K$ contains three vertices of degree three, and without loss of generality, these
vertices are $v_1$, $v_2$, and $v_4$.  For $i\in \{1,2,4\}$, let $u_i$ be the neighbor of $v_i$ not belonging to $K$.
Note that $u_1$ and $u_2$ are distinct and have no common neighbor by Lemma~\ref{lemma-disj5fv}.  Let $G_1$ be the graph obtained from $G-V(K)$ by adding the
edge $u_1u_2$.  Note that $G_1$ is triangle-free.

By Lemma~\ref{lemma-ext3cell}, any homomorphism from $G_1$ to a $5$-cycle would extend to a homomorphism from $G$ to a $5$-cycle.
Therefore, $G_1$ has no such homomorphism, and contains a $5/2$-critical subgraph $G_2$.
Since $G$ is a smallest counterexample, $p(G_2)\le 2$.

Since $G$ is $5/2$-critical, $G_2\not\subset G$, and it follows that $u_1u_2\in E(G_2)$.
Let $G_3$ be the graph obtained from $G_2-u_1u_2$ by adding the path $u_1v_1v_2u_2$.
Note that $p(G_3)=p(G_2)+2\le 4$. This contradicts Lemma~\ref{lemma-light}, since $v_3,v_4,v_5\not\in V(G_3)$, and thus $G\not\in \{G_3\}\cup P_3(G_3)$.
\end{proof}

For integers $k_1\ge \ldots\ge k_d\ge 0$,
a \emph{$(k_1,\ldots, k_d)$-vertex} of $G$ is a vertex of degree $d$ incident with a $k_1$-string, a $k_2$-string, \ldots, and a $k_d$-string.
A cell $K$ such that exactly a $k_1$-string, a $k_2$-string, \ldots, and a $k_d$-string have just one end in $K$ is called a \emph{$(k_1,\ldots, k_d)$-cell}.
If $x$ is either a $(k_1,\ldots, k_d)$-vertex or a $(k_1,\ldots, k_d)$-cell, we define the \emph{weight} $wt(x)=\sum_{i=1}^d k_i$.

\begin{corollary}\label{cor-no22x}
For any integer $k\ge 0$, no smallest counterexample contains a $(2,2,k)$-vertex.
\end{corollary}
\begin{proof}
Suppose that a smallest counterexample $G$ contains a $(2,2,k)$-vertex.
Lemma~\ref{lemma-opp5c} applied to the two strings incident with such a vertex would imply that $G$ contains
two distinct cells sharing an edge, contrary to Lemma~\ref{lemma-disj5fv}.
\end{proof}

For a graph $H$, let $Q'(H)\subseteq Q(H)$ denote the set of graphs obtained from $H$ by adding a vertex $z$, two $1$-strings joining $z$ to two adjacent vertices $x_1$ and $x_2$ in $H$
such that $x_1$ has degree at least three in $H$, and adding a $2$-string joining $z$ to another vertex $x_3\in V(H)$.

\begin{corollary}\label{cor-simplight}
If $G$ is a smallest counterexample, then $G$ does not contain a subgraph $H$ such that $G\in E_1(H)\cup E_2(H)$.
Furthermore, if $G\in Q(H)$, then $G\in Q'(H)$.
\end{corollary}
\begin{proof}
Suppose for a contradiction that $G\in E_1(H)\cup E_2(H)$.  Since $G$ contains neither a cell of degree three by Lemma~\ref{lemma-wtcell3}, nor a $(2,2,k)$-vertex by Corollary~\ref{cor-no22x},
the inspection of the graphs depicted in Figure~\ref{fig-excext} shows that $G$ contains a $7$-cycle $K=v_1\ldots v_7$ such that $v_2$, $v_5$, and $v_7$ have degree three and all other vertices of $K$ have degree two.
Let $e$ be the edge incident with $v_7$ that does not belong to $K$.
By Lemma~\ref{lemma-opp5c} applied to the $3$-string $v_2v_3v_4v_5$, both edges $v_1v_2$ and $v_5v_6$ are contained in cells.
By Lemma~\ref{lemma-disj5fv}, $K$ is an induced cycle, and it follows that $e$ is contained in two distinct cells.
This contradicts Lemma~\ref{lemma-disj5fv}.

If $G\in Q(H)$, then by Corollary~\ref{cor-no22x}, $G$ is obtained from $H$ by adding a vertex $z$, two $1$-strings joining $z$ to two distinct vertices $x_1$ and $x_2$ in $H$,
and a $2$-string joining $z$ to another vertex $x_3\in V(H)$.  Lemma~\ref{lemma-opp5c} implies
that $x_1$ and $x_2$ are adjacent, and by Lemma~\ref{lemma-wtcell3}, at least one of $x_1$ and $x_2$ has degree at least four in $G$, and thus degree at least $3$ in $H$.
Therefore, $G\in Q'(H)$.
\end{proof}

Also, Lemma~\ref{lemma-opp5c} trivially implies the following.

\begin{corollary}\label{obs-v3}
Let $v$ be a vertex of degree $3$ in a smallest counterexample.  Suppose that $v$ is not contained in a cell. Then $v$ has weight at most three,
and if $wt(v)=3$, then $v$ is a $(1,1,1)$-vertex.
\end{corollary}

Given a cycle $K$, a \emph{$K$-bad path} is a path of length three intersecting $K$ exactly in its ends.

\begin{lemma}\label{lemma-no6}
A smallest counterexample contains no $6$-cycles.
\end{lemma}
\begin{proof}
Suppose for a contradiction that $K=v_1v_2\ldots v_6$ is a $6$-cycle in a smallest counterexample $G$.
By Lemma~\ref{lemma-disj5fv}, $K$ is an induced cycle, no two vertices of $K$ have a common neighbor outside of $K$, and $G$ contains
no $K$-bad path.  Choose the labels of the vertices of $K$ so that $v_1$ has degree at least three.

Let $G_1$ be the graph obtained from $G$ by identifying the vertices $v_2$ and $v_6$ to a new vertex $w_1$,
and $v_3$ and $v_5$ to a new vertex $w_2$.  Observe that $G_1$ is triangle-free, and that it has no homomorphism to $C_5$.
Let $G_2$ be a $5/2$-critical subgraph of $G_1$.  Since $G_2\not\subseteq G$, at least one of $w_1$ or $w_2$ is a vertex of $G_2$.
By the minimality of $G$, we have $p(G_2)\le 2$.

If $w_2\not\in V(G_2)$, then let $G_3$ be the subgraph of $G$ obtained from $G_2$ by splitting $w_1$ back to $v_2$ and $v_6$
and adding the path $v_2v_1v_6$.  We have $p(G_3)\le p(G_2)+2\le 4$.  Since $G_3$ does not contain the vertices $v_3$ and $v_5$,
it follows that $G\not\in \{G_3\}\cup P_3(G_3)$, which contradicts Lemma~\ref{lemma-light}.  A symmetrical argument excludes the case that $w_1\not\in V(G_2)$.

Therefore, both $w_1$ and $w_2$ are vertices of $G_2$.  Let $G_3$ be the subgraph of $G$ obtained from $G_2$ by splitting $w_1$ back to $v_2$ and $v_6$,
splitting $w_2$ back to $v_3$ and $v_5$, and by adding $K$.
Let $a=|\{v_1,v_4\}\cap V(G_2)|$.  We have $n(G_3)=n(G_2)+4-a$ and $e(G_3)\ge e(G_2)+5-a$, hence
$p(G_3)\le p(G_2)+5(4-a)-4(5-a)=p(G_2)-a\le 2-a$.
By Lemma~\ref{lemma-light}, we have $G=G_3$ and $a=0$, and thus $v_1$ and $v_4$ have degree two in $G$.  This contradicts
the choice of the labels of $K$.
\end{proof}

We can now strengthen Lemma~\ref{lemma-disj5fv}.

\begin{lemma}\label{lemma-disj5}
If $H$ is a subgraph of a smallest counterexample containing at least two distinct cycles, then $n(H)\ge 10$.  In particular,
every two $5$-cycles in a smallest counterexample are vertex-disjoint.
\end{lemma}
\begin{proof}
Let $G$ be a smallest counterexample, and suppose that $H\subseteq G$ contains at least two distinct cycles.
Choose $H$ so that $n(H)$ is minimum, and subject to that, $e(H)$ is minimum.
If $H$ contains two vertex-disjoint cycles, then $n(H)\ge 10$ by Lemma~\ref{lemma-girth}.
By Observation~\ref{obs-mincyc}, $H$ either is a theta graph or consists of two cycles intersecting in exactly one vertex.
No such graph is $5/2$-critical, and thus $H\neq G$.  By Lemma~\ref{lemma-disj5fv}, we have $n(H)\ge 9$.
Note that $e(H)=n(H)+1$, and thus $p(H)=n(H)-4$.

Suppose for a contradiction that $n(H)=9$.  By Lemma~\ref{lemma-light} and Corollary~\ref{cor-simplight}, we have
$G\in P_2(H)\cup P_3(H)\cup Q'(H)$.  If $G\in P_2(H)$, then $p(G)=2$ and $n(G)=10$.  If $G\in Q'(H)$, then $p(G)=2$ and $n(G)=14$.
In both cases, this contradicts Lemma~\ref{lemma-nosmall}.

Hence, suppose that $G\in P_3(H)$.  If $H$ is the union of two $5$-cycles $v_1v_2v_3v_4v_5$ and $v_1v_6v_7v_8v_9$,
then observe that $G$ contains either a $k$-string with $k\ge 3$, or a $(2,2,1)$-vertex, contrary to Observation~\ref{obs-simp}
and Corollary~\ref{cor-no22x}.

If $H$ is a theta graph, observe that Lemma~\ref{lemma-no6} implies that $H$ is the union of paths of length $2$, $3$, and $5$ with common endvertices.
Let the paths be $v_1v_2v_3$, $v_1v_4v_5v_3$, and $v_1v_6v_7v_8v_9v_3$.  Let $G$ be obtained from $H$ by adding the path $v_ixyv_j$
for some $i,j\in \{1,\ldots, 9\}$.  By Lemma~\ref{lemma-no6}, we have $\{i,j\}\neq \{1,3\}$, and thus say $v_3$ has degree three in $G$.
If $\{i,j\}\cap \{4,5\}=\emptyset$, then Lemma~\ref{lemma-opp5c} implies that the path $v_2v_3v_9$ is contained in a cell, intersecting the
cell $v_1v_2v_3v_5v_4$.  This contradicts Lemma~\ref{lemma-disj5fv}, and thus we can assume that $i\in\{4,5\}$.  Since $G$ contains no $3$-strings or $4$-strings,
we have $j\in \{7,8\}$.  However, then $v_j$ is a $(2,2,1)$-vertex, which contradicts Corollary~\ref{cor-no22x}.
\end{proof}

\begin{lemma}\label{lemma-noeven}
A smallest counterexample contains no even cycles of length at most $8$.
\end{lemma}
\begin{proof}
By Lemmas~\ref{lemma-girth} and \ref{lemma-no6}, every even cycle in a smallest counterexample has length at least $8$.
For a contradiction, suppose that $K=v_1v_2\ldots v_8$ is an $8$-cycle in a smallest counterexample $G$.
By Lemma~\ref{lemma-disj5}, $K$ is an induced cycle and no two vertices of $K$ have a common neighbor outside of $K$.

Suppose that $G$ contains two distinct $K$-bad paths $P_1$ and $P_2$.
Let $H=K\cup P_1\cup P_2$.  Note that $n(H)\le 12$ and $e(H)=n(H)+2$, and thus $p(H)\le 4$.  By Lemma~\ref{lemma-light},
we have $G=H$ or $G\in P_3(H)$.
If $G\in P_3(H)$, then $p(G)\le 2$, which contradicts Lemma~\ref{lemma-nosmall}.
Suppose that $G=H$.  Let us distinguish two cases.
\begin{itemize}
\item If the paths $P_1$ and $P_2$ share an internal vertex, then since no two vertices of $K$ have a common neighbor outside of $K$,
$P_1$ intersects $P_2$ in exactly one edge.  By symmetry and Lemmas~\ref{lemma-girth} and \ref{lemma-no6}, we can assume that
$P_1=v_1xyv_3$ and $P_2=v_1xzv_i$ for some $i\in\{5,7\}$.  However, the cell $v_1xyv_3v_2$ contradicts Lemma~\ref{lemma-wtcell3}.
\item Therefore, $P_1$ and $P_2$ have no internal vertex in common.  By Lemmas~\ref{lemma-girth} and \ref{lemma-no6} and by symmetry, we can assume that $P_1=v_1xyv_i$ for some $i\in\{3,5\}$.
Since $G$ does not contain a $k$-string with $k\ge 3$, $P_2$ is incident with $v_j$ for some $j\in\{6,7,8\}$.
Hence, if $i=3$, then the cell $v_1v_2v_3yx$ contradicts Lemma~\ref{lemma-wtcell3}; consequently, $i=5$.
If $j\in \{6,8\}$, then by Lemma~\ref{lemma-girth} and Corollary~\ref{cor-no22x}, we conclude that $P_2=v_6wzv_8$.
However, then the cell $v_6v_7v_8zw$ contradicts Lemma~\ref{lemma-wtcell3}.  We conclude that $j=7$.  By a symmetrical argument,
$P_2$ is also incident with $v_3$.  However, then $G$ contains no $5$-cycles, and $P_2$ contradicts Lemma~\ref{lemma-opp5c}.
\end{itemize}

Therefore, $G$ contains at most one $K$-bad path.  Choose the labels of the vertices of $K$ so that $v_1$ has degree at least three
and if $G$ contains a $K$-bad path, then the path is incident with $v_1$.

Let $G_1$ be the graph obtained from $G$ by identifying $v_2$ with $v_8$ to a new vertex $w_2$, $v_3$ with $v_7$ to a new vertex $w_3$,
and $v_4$ with $v_6$ to a new vertex $w_4$.  Note that $G_1$ is triangle-free.
Observe that $G_1$ does not have a homomorphism to $C_5$,
and let $G_2$ be a $5/2$-critical subgraph of $G_1$.  Since $G$ is a smallest counterexample and $n(G_2)<n(G)$, we have $p(G_2)\le 2$.
Let $W=\{w_2,w_3,w_4\}\cap V(G_2)$ and let $G_3$ be obtained from $G_2$ by splitting all vertices of $W$ to their original vertices.
Let $a_1=|\{v_1\}\cap V(G_2)|$ and $a=|\{v_1,v_5\}\cap V(G_2)|$.
We distinguish several cases according to $W$.  Note that $G_2\not\subseteq G$, and thus $W\neq\emptyset$.
\begin{itemize}
\item If $|W|=3$, then let $G_4$ be the graph obtained from $G_3$ by adding $K$.
Note $n(G_4)=n(G_2)+5-a$ and $e(G_4)\ge e(G_2)+6-a$.  Consequently, $p(G_4)\le p(G_2)+1-a\le 3-a$.
By Lemma~\ref{lemma-light}, it follows that $G_4=G$. Since $v_1$ has degree at least three, we have $a\ge 1$.
By Lemma~\ref{lemma-light}, $a=1$ and $p(G_2)=2$,
i.e., $G_2\in \{E_1,E_2\}$.  We conclude that $n(G)\le 14$ and $p(G)\le 2$, which contradicts Lemma~\ref{lemma-nosmall}.

\item If $W=\{w_2,w_4\}$, then let $G_4$ be the graph obtained from $G_3$ by adding $K$.  Note that $n(G_4)=n(G_2)+6-a$,
$e(G_4)\ge e(G_2)+8-a$, and $p(G_4)\le p(G_2)-2-a\le 0$, which contradicts Lemma~\ref{lemma-light}.

\item If $W=\{w_2,w_3\}$, then let $G_4$ be the graph obtained from $G_3$ by adding the path $v_3v_2v_1v_8v_7$.
We have $n(G_4)=n(G_2)+3-a_1$, $e(G_4)\ge e(G_2)+3-a_1$ and $p(G_4)\le p(G_2)+3-a_1\le 5$.  Since $v_4,v_6\not\in V(G_4)$, observe that
$G\not\in \{G_4\}\cup P_2(G_4)\cup P_3(G_4)\cup Q'(G_4)$.  This contradicts Lemma~\ref{lemma-light} and Corollary~\ref{cor-simplight}.

\item The case $W=\{w_3,w_4\}$ is excluded symmetrically.

\item If $W=\{w_3\}$, then let $G_4$ be obtained from $G_3$ by adding $K$.
We have $n(G_4)=n(G_2)+7-a$, $e(G_4)=e(G_2)+8$ and $p(G_4)=p(G_2)+3-5a\le 5-5a$.
By Lemma~\ref{lemma-light} and Corollary~\ref{cor-simplight}, we have $a=0$ and $G\in\{G_4\}\cup P_2(G_4)\cup P_3(G_4)\cup Q'(G_4)$.
Let $R=E(G)\setminus E(G_4)$.

Since $a=0$, we have $v_1\not\in V(G_2)$, and thus $v_1$ has degree two in $G_4$.
Since $v_1$ has degree at least three in $G$, $v_1$ is incident with an edge of $R$, and thus $G\neq G_4$.
Since $G$ does not contain $3$-strings, another edge of $R$ is incident with $v_j$ for some $j\in\{4,5,6\}$.
Since $v_1$ and $v_j$ do not have a common neighbor outside of $K$, $G\not\in P_2(G_4)$.

Let us consider the case that $G\in Q'(G_4)$, that is, $G$ is obtained from $G_4$ by adding a vertex $z$,
two $1$-strings joining $z$ to two adjacent vertices $x_1$ and $x_2$ in $G_4$
such that $x_1$ has degree at least three in $G_4$, and a $2$-string joining $z$ to another vertex $x_3\in V(G_4)$.
Note that $\{v_1,v_j\}\subset\{x_1,x_2,x_3\}$.  Since $x_1$ has degree at least $3$ in $G_4$, and $v_1$, $v_2$, $v_8$ and $v_j$
have degree two, it follows that $x_1$ is adjacent to $v_j$ and $x_3=v_1$.

Therefore, if $G\in P_3(G_4)\cup Q'(G_4)$, then $v_1$ is incident with a $2$-string, and by Lemma~\ref{lemma-opp5c},
the path $v_2v_1v_8$ is contained in a cell.
This is not possible, since $v_2$ and $v_8$ have degree two and $K$ is an induced cycle.

\item If $W=\{w_2\}$, then
let $G_4$ be the graph obtained from $G_3$ by adding the path $v_2v_1v_8$.
We have $n(G_4)=n(G_2)+2-a_1$, $e(G_4)\ge e(G_2)+2-a_1$ and $p(G_4)\le p(G_2)+2-a_1\le 4$.
Since $v_3,v_4,v_6,v_7\not\in V(G_4)$, we have $G\not\in \{G_4\}\cup P_3(G_4)$.
This contradicts Lemma~\ref{lemma-light}.

\item The case $W=\{w_4\}$ is excluded symmetrically.
\end{itemize}
\end{proof}

\begin{corollary}\label{cor-idok}
If $G$ is a smallest counterexample and $G'$ is obtained from $G$ by identifying two of its vertices, then
$E_1\not\subseteq G'$ and $E_2\not\subseteq G'$.
\end{corollary}
\begin{proof}
This follows from Lemma~\ref{lemma-noeven}, since splitting any vertex in $E_1$ or $E_2$ results in a graph containing an $8$-cycle.
\end{proof}

We now need another result similar to Lemma~\ref{lemma-opp5c}.
\begin{lemma}\label{lemma-opp7c}
Let $G$ be a smallest counterexample, let $v\in V(G)$ have degree four, and let $P_1$, \ldots, $P_4$ be the strings incident with $v$.
If $P_1$ and $P_2$ have length at least two and $P_3$ and $P_4$ have length at least three, then $P_1\cup P_2$ is
contained either in a cell or a $7$-cycle.
\end{lemma}
\begin{proof}
Let $P'_1$, \ldots, $P'_4$ be the subpaths of $P_1$, \ldots, $P_4$, respectively, with common endvertex $v$, such that
$P'_1$ and $P'_2$ have length two and $P'_3$ and $P'_4$ have length three.
Let $v_1$, \ldots, $v_4$ be the endvertices of $P'_1$, \ldots, $P'_4$ distinct from $v$.  Suppose for a contradiction that $P'_1\cup P'_2$ is
contained neither in a cell nor a $7$-cycle.  Let $G_1$ be the graph obtained from $G$ by removing $v$ and the internal vertices of $P'_1$, \ldots, $P'_4$
and identifying $v_1$ with $v_2$ to a new vertex $z$.  Note that $G_1$ is triangle-free, and by Corollary~\ref{cor-idok}, $E_1\not\subseteq G_1$
and $E_2\not\subseteq G_1$.

Suppose that $G_1$ has a homomorphism $\psi$ to a $5$-cycle $C=c_1\ldots c_5$, where without loss of generality $\psi(z)=c_1$.
Set $\psi(v_1)=\psi(v_2)=c_1$ and choose $\psi(v)\in\{c_1,c_3,c_4\}$ distinct from $\psi(v_4)$ and $\psi(v_5)$.
Observe that $\psi$ extends to a homomorphism from $G$ to $C$, a contradiction.
Hence, $G_1$ has no such homomorphism.  Let $G_2$ be a $5/2$-critical subgraph of $G_1$.  Since $G_2\not\in\{E_1,E_2\}$, the minimality of $G$ implies
$p(G_2)\le 1$.

Note that $G_2$ contains $z$, and let $G_3$ be obtained
from $G_2$ by splitting $z$ back to $v_1$ and $v_2$ and adding the path $P'_1\cup P'_2$.  We have $p(G_3)=p(G_2)+4\le 5$.
However, since $G_3$ does not contain the internal vertices of $P'_3$ and $P'_4$, we have $G\not\in\{G_3\}\cup P_2(G_3)\cup P_3(G_3)\cup Q'(G_3)$.
This contradicts Lemma~\ref{lemma-light} and Corollary~\ref{cor-simplight}.
\end{proof}

Using this lemma, we can restrict the weight of vertices of degree $4$.

\begin{lemma}\label{lemma-wt4}
If $G$ is a smallest counterexample, then every vertex of $G$ of degree $4$ not contained in a cell has weight at most $6$.
\end{lemma}
\begin{proof}
Let $v\in V(G)$ be a $(k_1, \ldots, k_4)$-vertex.  By Observation~\ref{obs-simp}, $k_1,\ldots, k_4\le 2$.
Suppose for a contradiction that $v$ has weight at least $7$, and thus $k_1,k_2,k_3=2$ and $k_4\ge 1$.
For $1\le i\le 4$, let $v_i$ be the friend of $v$ joined to it by the $k_i$-string $P_i$.
If $k_4=2$, then let $G_1$ be the graph obtained from $G$ by removing $v$ and the internal vertices of the incident
strings.  It is easy to see that any homomorphism $\psi$ from $G_1$ to $C_5$ can be extended to a homomorphism from $G$,
by choosing $\psi(v)$ distinct from $\psi(v_1)$, \ldots, $\psi(v_4)$.  This contradicts the assumption that $G$ is $5/2$-critical.

Hence, we have $k_4=1$.
Since $v$ is not contained in a cell and by Lemma~\ref{lemma-no6}, $v_1$, $v_2$, $v_3$, and $v_4$ are pairwise distinct.
Furthermore, Lemma~\ref{lemma-opp7c} implies that $P_1\cup P_4$, $P_2\cup P_4$, and $P_3\cup P_4$ are contained in $7$-cycles,
and thus for $i\in \{1,2,3\}$, there exists a common neighbor $x_i$ of $v_4$ and $v_i$.  By Lemma~\ref{lemma-noeven},
$x_1$, $x_2$ and $x_3$ are pairwise different. By Lemma~\ref{lemma-disj5}, there is no path of length $3$ between $v_i$ and $v_4$
for $i\in \{1,2,3\}$.  Suppose that for all $1\le i<j\le 3$, there exists a path $P_{ij}$ of length at most $3$ between $v_i$ and $v_j$.
Let $H$ be the union of the paths $P_1$, \ldots, $P_4$, $P_{12}$, $P_{23}$, $P_{13}$, $v_1x_1v_4$, $v_2x_2v_4$, and $v_3x_3v_4$.
By Lemma~\ref{lemma-noeven}, no two of $v_1$, $v_2$, and $v_3$ have a common neighbor. Observe that this implies
$e(H)\ge n(H)+5$ (even if two of the paths $P_{12}$, $P_{23}$, and $P_{13}$ may share a vertex, or contain some of the vertices $x_1$, $x_2$, and $x_3$).
Note that $n(H)\le 21$, and thus $p(H)\le 1$.  This contradicts Lemma~\ref{lemma-light}.

Hence, by symmetry, we can assume that the distance between $v_1$ and $v_2$ is at least $4$.
Let $G_1$ be the graph obtained from $G$ by removing $v$ and the internal vertices of its incident strings, and identifying $v_1$, $v_2$, and $v_4$
to a new vertex $w$.  Note that $G_1$ is triangle-free.

Suppose that there exists a homomorphism $\psi$ from $G_1$ to a $5$-cycle $C=c_1\ldots c_5$, where without loss of generality $\psi(w)=c_1$.
Choose $\psi(v)\in \{c_3,c_4\}$ distinct from $\psi(v_3)$,
and observe that $\psi$ extends to a homomorphism from $G$ to $C$.  This is a contradiction,
and thus there exists no homomorphism from $G_1$ to $C_5$.  Let $G_2$ be a $5/2$-critical subgraph of $G_1$.
By the minimality of $G$, we have $p(G_2)\le 2$.

Since $G_2\not\subseteq G$, we have $w\in V(G_2)$.  Let $G'_3$ be the graph obtained from $G_2$ by
splitting $w$ back to the vertices $v_1$, $v_2$, and $v_4$.  Let $G_3$ be the graph obtained from $G'_3$
by adding the paths $v_1x_1v_4$, $v_2x_2v_4$, $P_1$, $P_2$, and $P_4$.
Let $a=|\{x_1,x_2\}\cap V(G_2)|$.
Note that $n(G_3)=n(G_2)+10-a$ and $e(G_3)\ge e(G_2)+12-a$, and thus $p(G_3)\le p(G_2)+2-a\le 4-a$.
Since the internal vertices of $P_3$ do not belong to $G_3$, we have $G\neq G_3$, and by Lemma~\ref{lemma-light},
$G=G_3+P_3$ and $a=0$.  Furthermore, $p(G_2)=2$, and thus $G_2\in\{E_1,E_2\}$.

By Corollary~\ref{cor-idok}, each of $v_1$, $v_2$ and $v_4$ has degree one in $G'_3$, and thus $v_1$ has degree three in $G$.
By Lemma~\ref{lemma-opp5c}, this implies that the edge $v_1x_1$ is contained in a cell.  Since $a=0$, $x_1$ has degree two in $G$,
and thus this cell shares the path $v_1x_1v_4$ with the $7$-cycle formed by $v_1x_1v_4$ and $P_1\cup P_4$.
This contradicts Lemma~\ref{lemma-disj5}.
\end{proof}

Next, we consider other types of vertices with large weight.

\begin{lemma}\label{lemma-111}
Let $v$ be a $(2,2,1,1)$-vertex or a $(1,1,1)$-vertex in a smallest counterexample $G$.
If $v$ is not contained in a cell, then $v$ has at most one friend of degree $3$ that is joined with $v$ by a $1$-string and is not contained in a cell.
\end{lemma}
\begin{proof}
Let $d=\deg(v)$ and let $v_1$, \ldots, $v_d$ be the friends of $v$, where $v_1$ and $v_2$ are joined to $v$ by $1$-strings.
Suppose for a contradiction that $v_1$ and $v_2$ have degree $3$ and are not contained in a cell.
For $i\in\{1,2\}$, let $x_i$ be the common neighbor of $v_i$ and $v$, and let $Y_i$ be the set of neighbors of $v_i$ distinct from $x_i$.  
By Lemma~\ref{lemma-noeven}, $Y_1$ and $Y_2$ are disjoint and have no common neighbors.

Let $G_1$ be the graph obtained from $G$ by removing $v$, all internal vertices
of its incident strings, and for $i\in \{1,2\}$, removing $v_i$ and identifying the vertices of $Y_i$ to a new vertex $w_i$.
Since neither $v_1$ nor $v_2$ is contained in a cell, it follows that $G_1$ is triangle-free.

Suppose that $G_1$ has a homomorphism $\psi$ to a $5$-cycle $C$.  If $\deg(v)=3$, then let $A$ be the set of vertices of $C$ non-adjacent to $\psi(v_3)$
(in particular, $\psi(v_3)$ is an element of $A$).
If $\deg(v)=4$, then let $A$ be the set of vertices of $C$ distinct from $\psi(v_3)$ and $\psi(v_4)$.
In either case, $|A|=3$.
For each $i\in \{1,2\}$ and $y\in Y_i$, set $\psi(y)=\psi(w_i)$.  Choose $\psi(v)\in A$ distinct from $\psi(w_1)$ and $\psi(w_2)$.
Observe that this choice ensures that $\psi$ can be extended to a homomorphism from $G$ to $C$.
Since $G$ does not have a homomorphism to $C_5$, this is a contradiction.  Therefore, $G_1$ has no homomorphism to $C_5$; let $G_2$ be a $5/2$-critical
subgraph of $G_1$.  Note that $p(G_2)\le 2$ by the minimality of $G$.

Since $G$ is $5/2$-critical, we have $G_2\not\subseteq G$, and thus at least one of $w_1$ or $w_2$ is a vertex of $G_2$.
Suppose that exactly one of $w_1$ or $w_2$, say $w_1$, is in $V(G_2)$.  Let $G_3$ be the graph obtained from $G_2$ by splitting $w_1$ back to the vertices of $Y_1$,
and by adding $v_1$ and the edges between $v_1$ and $Y_1$.  We have $p(G_3)=p(G_2)+2\le 4$, which contradicts Lemma~\ref{lemma-light} since
$G\not\in P_3(G_3)$.

So both $w_1$ and $w_2$ are vertices of $G_2$.  Let $G_3$ be the graph obtained from $G_2$ by splitting $w_i$ back to the vertices of $Y_i$ and
adding $v_i$ and the edges between $v_i$ and the vertices of $Y_i$ for $i\in\{1,2\}$, and by adding the path $v_2x_2vx_1v_1$.
We have $p(G_3)=p(G_2)+3\le 5$.  Note that $G\neq G_3$.  If $\deg(v)=4$, then observe that $G\not\in P_2(G_3)\cup P_3(G_3)\cup Q'(G_3)$,
which contradicts Lemma~\ref{lemma-light} and Corollary~\ref{cor-simplight}.

Suppose that $\deg(v)=3$.  Since $v$ and $v_3$
are joined by a $1$-string in $G$, and $x_1$, $x_2$, and $v$ have degree two in $G_3$, we conclude that $G\not\in P_3(G_3)\cup Q'(G_3)$.
Hence, by Lemma~\ref{lemma-light} and Corollary~\ref{cor-simplight}, we have $G_2\in\{E_1,E_2\}$ and $G\in P_2(G_3)$.
Note that $p(G)=p(G_3)-3=2$ and $n(G)=n(G_2)+8=18$.  This contradicts Lemma~\ref{lemma-nosmall}.
\end{proof}

Let us also give another similar result concerning $(2,1,1,1)$-vertices.

\begin{lemma}\label{lemma-2111}
Let $v$ be a $(2,1,1,1)$-vertex in a smallest counterexample $G$.
If $v$ is not contained in a cell, then $v$ has at most two friends that are $(1,1,1)$-vertices and are not contained in a cell.
\end{lemma}
\begin{proof}
Let $v_1$, \ldots, $v_4$ be the friends of $v$, where $v_4$ is joined by the $2$-string.
Suppose for a contradiction that $v_1$, $v_2$, and $v_3$ have degree $3$ and are not contained in a cell.
For $i\in\{1,2,3\}$, let $x_i$ be the common neighbor of $v_i$ and $v$, and let $Y_i$ be the set of neighbors of $v_i$ distinct from $x_i$.  

Let $G_1$ be the graph obtained from $G$ by removing $v$, all internal vertices
of the incident strings, and for $i\in \{1,2,3\}$, removing $v_i$ and identifying the vertices of $Y_i$ to a new vertex $w_i$.
By Lemma~\ref{lemma-noeven}, $Y_1$, $Y_2$, and $Y_3$ are disjoint and have no common neighbors.
Since $v_1$, $v_2$, and $v_3$ are $(1,1,1)$-vertices, $Y_1\cup Y_2\cup Y_3$ is an independent set in $G$.
Since $v_1$, $v_2$, and $v_3$ are not contained in a cell, it follows that $G_1$ is triangle-free.

Suppose that $G_1$ has a homomorphism $\psi$ to a $5$-cycle $C$.
For each $i\in \{1,2,3\}$ and $y\in Y_i$, set $\psi(y)=\psi(w_i)$.  Choose $\psi(v)\in V(C)$ distinct from $\psi(w_1)$, $\psi(w_2)$, $\psi(w_3)$, and $\psi(v_4)$.
Observe that this choice ensures that $\psi$ can be extended to a homomorphism of $G$ to $C$.
Since $G$ does not have a homomorphism to $C_5$, this is a contradiction.  Therefore, $G_1$ has no homomorphism to $C_5$; let $G_2$ be a $5/2$-critical
subgraph of $G_1$.  Note that $p(G_2)\le 2$.

Let $W=\{w_1,w_2,w_3\}$.
Since $G$ is $5/2$-critical, we have $G_2\not\subseteq G$, and thus $W\cap V(G_2)\neq\emptyset$.
If $|W\cap V(G_2)|=1$, say $W\cap V(G_2)=\{w_1\}$,
then let $G_3$ be the graph obtained from $G_2$ by splitting $w_1$ back to the vertices of $Y_1$,
and by adding $v_1$ and the edges between $v_1$ and $Y_1$.  We have $p(G_3)=p(G_2)+2\le 4$, which contradicts Lemma~\ref{lemma-light} since
$G\not\in \{G_3\}\cup P_3(G_3)$.

If $|W\cap V(G_2)|=2$, say $W\cap V(G_2)=\{w_1, w_2\}$, then let $G_3$ be the graph obtained from $G_2$ by splitting $w_i$ back to the vertices of $Y_i$ and
adding $v_i$ and the edges between $v_i$ and the vertices of $Y_i$ for $i\in\{1,2\}$, and by adding the path $v_2x_2vx_1v_1$.
We have $p(G_3)=p(G_2)+3\le 5$.  Observe that $G\not\in \{G_3\}\cup P_2(G_3)\cup P_3(G_3)\cup Q'(G_3)$, and thus $G_3$ contradicts Lemma~\ref{lemma-light} and Corollary~\ref{cor-simplight}.

Therefore, $W\subset V(G_2)$. Let $G_3$ be the graph obtained from $G_2$ by splitting $w_i$ back to the vertices of $Y_i$ and
adding $v_i$ and the edges between $v_i$ and the vertices of $Y_i$, and the paths $v_ix_iv$ for $i\in\{1,2,3\}$.
We have $p(G_3)=p(G_2)+2\le 4$.  By Lemma~\ref{lemma-light}, we conclude that $p(G_3)=4$ and $G\in P_3(G_3)$.  Hence, $p(G_2)=2$ and
$G_2\in\{E_1,E_2\}$.

Since the vertices of $Y_1\cup Y_2\cup Y_3$ are distinct, non-adjacent and have no common neighbors other than $v_1$, $v_2$, and $v_3$,
it follows that the distance between any two of $w_1$, $w_2$, and $w_3$ in $G_2$ is at least three.  This is only possible if $G_2=E_1$ and
$w_1$, $w_2$ and $w_3$ have degree two in $G_2$, and the corresponding graph $G$ (up to the position of the vertex $v_4$) is
depicted in Figure~\ref{fig-2111}.  However, then at most one of $v_1$, $v_2$, and $v_3$ is a $(1,1,1)$-vertex, which is a contradiction.
\end{proof}

\begin{figure}
\centering{\includegraphics{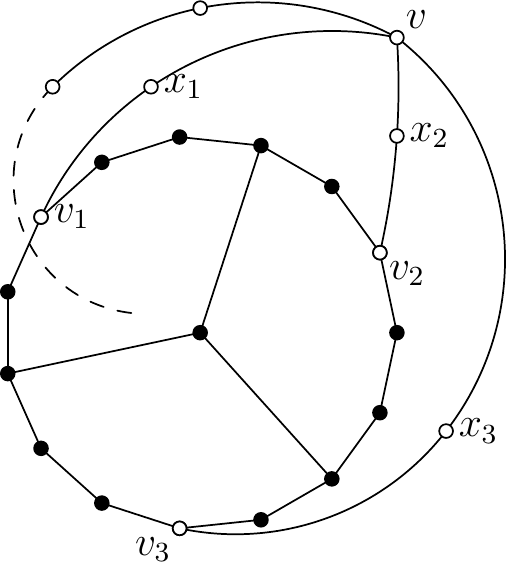}}
\caption{A configuration from the proof of Lemma~\ref{lemma-2111}.}\label{fig-2111}
\end{figure}

We also need to bound the weight of cells of degree $4$.

\begin{lemma}\label{lemma-wtcell4}
If $K$ is a $(2,k_2,k_3,k_4)$-cell in a smallest counterexample $G$, then $k_4=0$.
\end{lemma}
\begin{proof}
Let $K=v_1v_2v_3v_4v_5$, and suppose for a contradiction that $k_2,k_3,k_4\ge 1$.  Let $P_1$, \ldots, $P_4$ be the paths in $G$ with exactly one end
in $K$ and with internal vertices of degree two, where $P_1$ has length three and $P_2$, $P_3$, and $P_4$ have length two.

Suppose that $K$ contains three vertices of degree two.  By Observation~\ref{obs-simp}, we can assume that each of $v_1$ and $v_3$ is incident
with at least one of the paths $P_1$, \ldots, $P_4$.  By Lemmas~\ref{lemma-opp5c} and~\ref{lemma-disj5}, we can assume that $P_1$ and $P_2$
are incident with $v_1$, and $P_3$ and $P_4$ are incident with $v_3$.  However, by Lemma~\ref{lemma-opp7c}, the paths $P_2$ and $v_1v_2v_3$ belong to a cell or a $7$-cycle,
whose union with $K$ contradicts Lemma~\ref{lemma-disj5}.

Hence, $K$ contains at most two vertices of degree two.
Let us first consider the case that $K$ contains two vertices of degree two. Lemmas~\ref{lemma-opp5c} and~\ref{lemma-disj5} imply that they are non-adjacent.
By symmetry, we can assume that $v_1$ is incident with two of the paths $P_1$, \ldots $P_4$, that $v_2$ has degree two, and that $v_3$ is incident with one of the paths.
Let $z_1$ and $z_2$ be the neighbors of $v_1$ that do not belong to $K$, and let $z_3$ be the neighbor of $v_3$ not belonging to $K$.

Suppose that $v_5$ is incident with one of the paths $P_1$, \ldots $P_4$, and let $z_5$ be its neighbor not belonging to $K$.  Let $G_1$ be the graph obtained from $G$
by removing $K$ and $z_3$, identifying $z_1$ with $z_2$ to a new vertex $w$, and adding the edge $wz_5$.  By Lemma~\ref{lemma-disj5}, $G_1$ is triangle-free.
By Lemma~\ref{lemma-ext3cell}, any homomorphism from $G_1$ to $C_5$ would extend to a homomorphism from $G$,
and thus no such homomorphism exists.  Let $G_2$ be a $5/2$-critical subgraph of $G_1$.  By the minimality of $G$, we have $p(G_2)\le 2$.
Let $G_3$ be obtained from $G_2$ by removing the edge $wz_5$ (if present), splitting $w$ back to $z_1$ and $z_2$,
adding the path $z_1v_1z_2$, and if $z_5\in V(G_2)$, adding the path $v_1v_5z_5$.
We have $p(G_3)\le p(G_2)+3\le 5$.  By Lemma~\ref{lemma-light} and Corollary~\ref{cor-simplight},
we have $G_2\in\{E_1,E_2\}$ and (since $v_2,v_3,v_4,z_3\not\in V(G_3)$) $G\in Q'(G_3)$.  Note that $p(G)=p(G_3)-3=2$ and $n(G)=n(G_2)+8=18$, which contradicts Lemma~\ref{lemma-nosmall}.

Hence, if $K$ contains two vertices of degree two, we can assume that $v_5$ has degree two, and $v_4$ has a neighbor $z_4$ not belonging to $K$.
Let $G_1$ be obtained from $G$ by removing $K$, identifying $z_1$ with $z_2$ to a new vertex $w$, and adding the edge $z_3z_4$.
By Lemma~\ref{lemma-disj5}, $G_1$ is triangle-free.  By Lemma~\ref{lemma-ext3cell}, any homomorphism from $G_1$ to $C_5$ would extend to a homomorphism from $G$,
and thus no such homomorphism exists.  Let $G_2$ be a $5/2$-critical subgraph of $G_1$, and note that $p(G_2)\le 2$.  Let $G_3$ be obtained from $G_2$ by
\begin{itemize}
\item if $w\in V(G_2)$, splitting $w$ back to $z_1$ and $z_2$ and adding the path $z_1v_1z_2$, and
\item if $z_3z_4\in E(G_2)$, removing the edge $z_3z_4$ and adding the path $z_3v_3v_4z_4$, and
\item if both $w\in V(G_2)$ and $z_3z_4\in E(G_2)$, also adding the path $v_1v_2v_3$.
\end{itemize}
Observe that $p(G_3)\le p(G_2)+2\le 4$ and that $G\not\in \{G_3\}\cup P_3(G_3)$, contrary to Lemma~\ref{lemma-light}.

Therefore, $K$ contains only one vertex of degree two, and the ends of $P_1$, \ldots, $P_4$ in $K$ are pairwise distinct.
By symmetry, we can assume that $P_1$ is incident with $v_1$, $P_2$ with $v_2$, and $P_3$ with $v_3$.  For $i\in\{2,3\}$, let $z_i$ be the neighbor of $v_i$ not in $K$.
Let $z_1$ and $z_4$ be the endvertices of $P_1$ and $P_4$, respectively, that do not belong to $K$.

Suppose that $P_4$ is incident with $v_4$.  Let $G_1$ be the graph obtained from $G$ by removing $K$ and the internal vertices of $P_1$ and $P_4$,
adding the edge $z_2z_3$, and adding a path $z_1y_1y_2z_4$.  By Lemma~\ref{lemma-disj5}, $G_1$ is triangle-free.
A straightforward case analysis shows that any homomorphism from $G_1$ to $C_5$ would extend to a homomorphism from $G$,
and thus no such homomorphism exists.  Let $G_2$ be a $5/2$-critical subgraph of $G_1$.
Let $G_3$ be obtained from $G_2$ by removing $y_1$ and $y_2$ (if present in $G_2$), if $G_2$ contains the edge $z_2z_3$, removing this edge and adding the path $z_2v_2z_3v_3$,
and if both $y_1\in V(G_2)$ and $z_2z_3\in E(G_2)$, adding the paths $P_4$ and $v_3v_4$.  Observe that $p(G_3)\le p(G_2)+2\le 4$ and $G\not\in \{G_3\}\cup P_3(G_3)$, contrary to Lemma~\ref{lemma-light}.

Finally, suppose that $P_4$ is incident with $v_5$.  By the result of the previous paragraph and symmetry, we can assume that $P_4$ is not a part of a $2$-string.
Let $G_1$ be the graph obtained from $G$ by removing $K$ and the internal vertices of $P_1$ and $P_4$,
adding the edge $z_2z_3$, and adding a path $z_1yz_3$.  By Lemma~\ref{lemma-disj5}, $G_1$ is triangle-free.
A straightforward case analysis shows that any homomorphism from $G_1$ to $C_5$ would extend to a homomorphism from $G$,
and thus no such homomorphism exists.  Let $G_2$ be a $5/2$-critical subgraph of $G_1$.
Let $G_3$ be obtained from $G_2$ by removing $y$ (if present in $G_2$), if $G_2$ contains the edge $z_2z_3$, removing this edge and adding the path $z_2v_2z_3v_3$,
and if both $y\in V(G_2)$ and $z_2z_3\in E(G_2)$, adding the path $P_1$ and $K$.  If $z_2z_3\not\in E(G_2)$, then since $G_2\not\subseteq G$ we have $y\in V(G_2)$,
and $p(G_3)=p(G_2)+3\le 5$.  Since $G\not\in \{G_3\}\cup P_2(G_3)\cup P_3(G_3)\cup Q'(G_3)$, this contradicts Lemma~\ref{lemma-light} and Corollary~\ref{cor-simplight}.
If $y\not\in V(G_2)$, then $z_2z_3\in E(G_2)$ and $p(G_2)=p(G_3)+2\le 4$.  Since $G\not\in \{G_3\}\cup P_3(G_3)$, this contradicts Lemma~\ref{lemma-light}.
Finally, let us consider the case that $y\in V(G_2)$ and $z_2z_3\in E(G_2)$.
Observe that $p(G_3)\le p(G_2)+2\le 4$.  By Lemma~\ref{lemma-light}, it follows that $G\in \{G_3\}\cup P_3(G_3)$.  This is a contradiction, since
$P_4\not\subseteq G_3$ and $P_4$ is not a part of a $2$-string.
\end{proof}

\section{Discharging}\label{sec-disch}

We now finish the proof of the main result by discharging.

\begin{proof}[Proof of Theorem~\ref{thm-main}]
If Theorem~\ref{thm-main} is false, then there exists a smallest counterexample $G$.  Firstly, let us assign charge $ch_0(v)=10-4\deg(v)$ to each vertex of $G$.  Note that the sum of the charges is $10n(G)-8e(G)=2p(G)$.
Vertices of degree two have charge $2$.  Next, each vertex $v$ of degree two sends $1$ unit of charge to the ends of the string containing $v$, obtaining an altered charge $ch_1$
where $ch_1(v)=0$ for vertices of degree two and $ch_1(v)=10-4\deg(v)+wt(v)$ for any vertex $v$ of degree at least three.
Next, each vertex contained in a cell sends all of its charge to the cell, obtaining the second altered charge $ch_2$.
Hence $ch_2(v)=0$ for all vertices of degree two and all vertices contained in cells.
For a cell $K$, if we let $X$ denote the set of vertices of degree at least three in $K$, then
\begin{eqnarray*}
ch_2(K)&=&\sum_{v\in X} (10-4\deg(v)+wt(v))\\
&=&10|X|-(8|X|+4\deg(K))+2(5-|X|)+wt(K)\\
&=&10-4\deg(K)+wt(K).
\end{eqnarray*}

Let us now move the charge according to the following rule, obtaining the final charge $ch_3$.
Let $v_1v_2v_3$ be a $1$-string in $G$, where $v_3$ is not contained in a cell.  If
\begin{itemize}
\item $v_1$ is contained in a cell $K$, or
\item $v_1$ has degree at least five, or
\item $v_1$ has degree four and is incident with a $0$-string, or
\item $v_1$ has degree four and $v_3$ is a $(1,1,1)$-vertex,
\end{itemize}
then $v_3$ sends $1/2$ to $v_1$ (or the cell $K$ in the first case).

We will show that every cell and vertex has final charge at most $0$.
Let us first consider the final charge of a vertex $v\in V(G)$.
Clearly, $ch_3(v)\le 0$ if $v$ either has degree two or is contained in a cell.
Suppose that $v$ has degree at least three and is not contained in a cell.

If $v$ has degree $3$, then by Lemma~\ref{lemma-opp5c}, $v$ is not incident with a $2$-string.  If $wt(v)\le 2$,
then $ch_2(v)\le 0$ and $ch_3(v)\le ch_2(v)$.  If $v$ is a $(1,1,1)$-vertex, then $ch_2(v)=1$ and by Lemma~\ref{lemma-111}, $v$ sends $1/2$ to at least two of its friends,
and thus $ch_3(v)\le 0$.  Therefore, every vertex of degree three has non-positive charge.

Let us now consider the case that $v$ has degree at least $4$.
Let $k$ denote the number of $1$-strings incident with $v$.
If $v$ has degree at least $5$, then by Observation~\ref{obs-simp} we have $wt(v)\le 2\deg(v)-k$,
and thus $ch_2(v)\le 10-2\deg(v)-k\le -k$.  The vertex $v$ receives $1/2$ at most $k$ times, and thus its final charge is $ch_3(v)\le -k/2\le 0$.

Suppose that $v$ has degree $4$.  If $v$ is incident with a $0$-string, then $wt(v)\le 6-k$, and thus $ch_2(v)\le -k$ and $ch_3(v)\le -k/2\le 0$.
If $v$ is not incident with a $0$-string, then by Lemma~\ref{lemma-wt4}, $v$ is a $(2,2,1,1)$, $(2,1,1,1)$ or a $(1,1,1,1)$-vertex.
If $v$ is a $(1,1,1,1)$-vertex, then $ch_2(v)=-2$ and $ch_3(v)\le -2+k/2\le 0$.  If $v$ is a $(2,1,1,1)$-vertex, then by Lemma~\ref{lemma-2111}, $v$ receives charge from at most two friends,
and $ch_3(v)\le ch_2(v)+1=0$.  If $v$ is a $(2,2,1,1)$-vertex, then let $vx_1v_1$ and $vx_2v_2$ be the $1$-strings incident with $v$.  Since $v$ is not contained in a cell,
Lemma~\ref{lemma-opp7c} implies that $G$ contains a path $v_1y_1y_2v_2$.  Consequently, at most one of $v_1$ and $v_2$ is a $(1,1,1)$-vertex.  If neither $v_1$ nor $v_2$ is a $(1,1,1)$-vertex not
contained in a cell,
then $ch_3(v)=ch_2(v)=0$.  If say $v_1$ is a $(1,1,1)$-vertex, then $y_2$ has degree at least three, and thus $v_2$ is incident with a $0$-string.
Furthermore, by Lemma~\ref{lemma-111}, if $v_1$ is not contained in a cell, then $v_2$ is either contained in a cell or has degree at least $4$.
We conclude that if $v$ receives $1/2$ from $v_1$, then it sends $1/2$ to $v_2$, and $ch_3(v)=ch_2(v)+1/2-1/2=0$.

Therefore, every vertex of $G$ has non-positive charge.  Let us now consider the final charge of a cell $K$.  Let $k$ denote the number of $1$-strings with exactly one end in $K$.
By Lemma~\ref{lemma-wtcell3}, we have $\deg(K)\ge 4$.  If $\deg(K)=4$, then by Lemma~\ref{lemma-wtcell4}, $K$ is either a $(1,1,1,1)$-cell, or a $(k_1,k_2,k_3, 0)$-cell for some integers
$k_1\ge k_2\ge k_3\ge 0$.
In the former case, we have $ch_2(K)=-2=-k/2$.  In the latter case, $wt(K)\le 6-k$, and thus $ch_2(K)\le -k$.  If $\deg(K)\ge 5$, then $wt(K)\le 2\deg(K)-k$
and $ch_2(K)\le 10-2\deg(K)-k\le -k$.  In all cases, $ch_3(K)\le ch_2(K)+k/2\le 0$.

Therefore, $$2p(G)=\sum_v ch_0(v)=\sum_v ch_3(v)+\sum_K ch_3(K)\le 0,$$
which contradicts the assumption that $p(G)\ge 2$.
\end{proof}

\bibliographystyle{siam}
\bibliography{circul}

\end{document}